\documentclass[11pt]{article}
\usepackage{amssymb}
\usepackage{amsmath}
\usepackage{bbm}
\usepackage{float}
\usepackage{multirow}	
\usepackage{enumerate}
\usepackage[dvips]{graphics}
\usepackage{color}
\usepackage{ifpdf}
\ifpdf
\usepackage{epstopdf}
\fi

\usepackage{amsthm,amssymb,latexsym,amsmath,color,mathrsfs,caption}
\usepackage[all]{xy}
\usepackage[utf8]{inputenc}
\usepackage{multirow}
\usepackage{makecell}
\usepackage{marginnote}
\usepackage{float}
\usepackage{aliascnt}
\usepackage{hyperref}
\hypersetup{
colorlinks=true,
linkcolor=blue,
filecolor=magenta, 
urlcolor=cyan,
}
\usepackage{amsbsy}
\numberwithin{equation}{section}

\usepackage{amsthm}
\newtheorem{theorem}{Theorem}[section]
\newtheorem{lemma}{Lemma}[section]
\newtheorem{proposition}{Proposition}[section]
\newtheorem{definition}{Definition}[section]

\def\qand{{\quad\text{and}\quad}}

\newcommand{\p}{{\mathbb{P}}}
\newcommand{\pn}{{\mathbb{P}^n}}

\newcommand{\op}[1]{{\mathcal O}_{\mathbb{P}^{#1}}}

\newcommand{\tp}[1]{{\rm T}{\mathbb{P}^{#1}}}

\newcommand{\Ext}{\operatorname{Ext}}
\newcommand{\Hom}{\operatorname{Hom}}
\newcommand{\sing}{\operatorname{Sing}}

\DeclareMathOperator{\coker}{coker}

\def\sD{{\mathscr{D}}}

\def\thanksAuthorx{Universidad del Sin\'u-El\'ias Bechara Zain\'um, Facultad de Ciencias e Ingenier\'ias, Departamento de Ciencias B\'asicas, Carrera 1w No. 38-153, Barrio Juan XXIII, Monter\'ia-C\'ordoba, Colombia. email: {\tt hugogaleano@unisinu.edu.co} Orcid: https://orcid.org/0000-0002-6588-1406}

\def\thanksAuthory{Universidad de C\'ordoba, Facultad de Ciencias B\'asicas, Departamento de Matemáticas y Estadística, Carrera 6 No. 77-305, Monter\'ia-C\'ordoba, Colombia. email: {\tt ochaljubzapa@correo.unicordoba.edu.co} Orcid: https://orcid.org/0009-0005-5573-5260}

\begin{document}
	\title{Codimension one distributions of degree $3$ on the three-dimensional projective space}
	
	\author{ {\sc Hugo Galeano }\thanks{\thanksAuthorx}
 \qand {\sc Orlando Chaljub}\thanks{\thanksAuthory}
           }

	\date{}
	
	\maketitle

\begin{abstract}\noindent 
We make a classification of codimension one degree 3 distributions on the projective three space, giving possible Chern classes of the tangent sheaf and describing de zero and one dimensional components of the singular scheme of the distribution. Also, we show the existence and describe  some moduli spaces of such distributions, using the concept of stability of the tangent sheaf.
\end{abstract}



\section{Introduction}\label{sec:intro}


Let $X$ be a smooth projective variety over an algebraically closed field $k$. A codimension $r$ distribution $\mathscr{D}$ on  $X$ is given by an exact sequence

\begin{equation}\label{dis}
\mathscr{D}: 0\to T_{\mathscr{D}} \overset{\phi}{\to} TX \overset{\pi}\to N_{\mathscr{D}} \to 0,
\end{equation}

where $T_{\mathscr{D}}$ is a coherent sheaf of rank $s:=\text{dim} (X)-r$, and $N_{\mathscr{D}}$ is a torsion-free sheaf. The sheaves $T_{\mathscr{D}}$ and $N_{\mathscr{D}}$ are called the \emph{tangent} and the \emph{normal}  sheaves of $\mathscr{D}$, respectively. The distribution $\mathscr{D}$ is called a \emph{foliation} if it is integrable in Frobeniu's sense, that is, is invariant under the Lie bracket, $[T_{\mathscr{D}}, T_{\mathscr{D}}]\subset T_{\mathscr{D}}$.\\

The theory of distributions and foliations has its origins in the nineteenth century through works by Grasmann, Jacobi, Clebsh, Cartan, and Frobenius. The qualitative study of polynomial differential equations was initiated in classical works by
Poincaré, Darboux, and Painlevé, see for instance \cite{darboux1878memoire, Painleve, Poincaré1891161}. In modern terminology the classification and description of spaces of codimension one foliation which are given by polynomial differential 1-forms was initiated by Jouanolou in degrees 0 and 1 see \cite{jouanolou2006equations}, and Cervau and Lins Neto in degree 2, see \cite{CLN}.\\

Recently, a systematic study of distributions, not necessarily integrable, was initiated by the authors of \cite{Calvo-Andrade20209011}. They classify codimension one distribution in terms of invariants of the tangent sheaf and the singular scheme of the distribution. The authors describe the moduli space of distributions using Grothendieck's Quot-scheme for the tangent bundle, they showed that there exists a quasi-projective variety, a moduli space, $\mathcal{D}^{P}$ that parametrizes isomorphism classes of distributions on $X$, where $P$ is fixed is the Hilbert Polynomial of the tangent sheaf $ T_{\mathscr{D}}$. They describe some moduli spaces for codimension one distribution on $\mathbb{P}^3$. Also, the author of \cite{Galeano2022} made a complete picture of codimension one degree 2 distributions on $\p^3$, giving all possible Chern classes of the tangent sheaves, and using differents tecniques, they show the existence of such distribution with this invariants, they also describe some moduli spaces for this distributions. Recently, the authors of \cite{correa2022moduli} gave the full picture of  moduli space of codimension one distributions of degree 1 on $\p^3$.\\

In this work, we give a picture of possible Chern classes of the tangent sheaf of codimension one distribution of degree 3 on $\p^3$, we show existence, describe moduli spaces and stability of the tangent sheaf. The next are the main results of the work.

\begin{theorem}
If $T_\sD$ is the tangent sheaf of a codimension one degree 3 distribution then the possible Chern classes of $T_\sD$ are
    \begin{table}[H] \centering
\begin{tabular}{c c c}
\hline
 deg(C) & $c_2(T_\sD)$ &$c_3(T_\sD)$   \\ \hline \hline
 0 & 11 & 51 \\
 1 & 10 &  42 \\
 2 & 9  &  29, 31, 33, 35\\ 
 3 & 8 & 0, 2, 4, 6, 8, 10, 12, 14, 16, 18, 20, 22, 24, 26, 28, 30  \\
 4 & 7 & 1, 3, 5, 7, 9, 11, 13, 15, 17, 19, 21, 23, 25, 27 \\
 5 & 6 & 0, 2, 4, 6, 8, 10, 12, 14, 16, 18, 20, 22, 24, 26 \\
 6 & 5 & 1, 3, 5, 7, 9, 11, 13, 15, 17, 19, 21, 23, 25 \\
 7 & 4 & 0, 2, 4, 6, 8, 10, 12, 14, 16 \\
 8 & 3 & 1, 3, 5, 7, 9, 11, 13, 15 \\
 9 & 2 & 0, 2, 4, 6, 8 \\
 10 & 1 & 1, 3 \\
 11 & 0 & 0 \\
 13 & -2 & 0 \\
 
\end{tabular}
\caption{Possible Chern classes of codimension one distribution of degree 3; the first column describes the degree of a generic point in the irreducible component of the Hilbert scheme that contains $\sing_1(\sD)$. 
} 
\label{deg 3 table}
\end{table}
\end{theorem}

\begin{theorem}
    Let $T_\sD$ be the tangent sheaf of a codimension one degree 3 distribution. If  $3\leq c_2(T_\sD)\leq 11$,  then $T_\sD$ is stable.
\end{theorem}

\begin{theorem}
    There exist codimension one distributions with tangent sheaf having Chern classes $(-1, 2, 2)$, $(-1, 3, 5)$, and $(-1, 3, 7)$.
\end{theorem}

\begin{theorem}
    The moduli space $\mathcal{D}^{st}(3, 3, 5)$ of codimension one distribution of degree 3 with Chern classes $(-1, 3, 5)$ and stable and general tangent sheaf, is  an irreducible quasi projective variety, of dimension
    \[\dim\mathcal{D}^{st}(3, 3, 5)= 42.\]
\end{theorem}





\section{Codimension One Distributions}\label{sec:contprob}
Let $X$ be a smooth projective variety over an algebraically closed field $k$. A codimension $r$ distribution $\mathscr{D}$ on  $X$ is given by an exact sequence

\begin{equation}\label{dis}
\mathscr{D}: 0\to T_{\mathscr{D}} \overset{\phi}{\to} TX \overset{\pi}\to N_{\mathscr{D}} \to 0,
\end{equation}

where $T_{\mathscr{D}}$ is a coherent sheaf of rank $s:=\text{dim} (X)-r$, and $N_{\mathscr{D}}$ is a torsion-free sheaf. The sheaves $T_{\mathscr{D}}$ and $N_{\mathscr{D}}$ are called the \emph{tangent} and the \emph{normal}  sheaves of $\mathscr{D}$, respectively. Since $TX$ is locally free and $N_{\mathscr{D}}$ is torsion-free then $T_{\mathscr{D}}$ must be reflexive, see \cite[Proposition 1.1]{hartshorne1980stable}.\\

Two distributions $$\mathscr{D}: 0\to T_{\mathscr{D}} \overset{\phi}{\to} TX \overset{\pi}\to N_{\mathscr{D}} \to 0,$$ and $$\mathscr{D}': 0\to T'_{\mathscr{D}} \overset{\phi'}{\to} TX \overset{\pi'}\to N'_{\mathscr{D}} \to 0,$$

are said to be isomorphic if there exists an isomorphism $\beta: T_{\mathscr{D}}\to T_{\mathscr{D} '}$, such that $\phi'\circ \beta=\phi$. According to \cite{Calvo-Andrade20209011} every isomorphism class of codimension $r$ distribution on $X$ induces an element in the projective space
\[\mathbb{P}H^0(\wedge^{n-r}TX\otimes \text{det} (T_{\mathscr{D}})^{\vee})=\mathbb{P}H^0(\Omega^r_{X}\otimes \text{det} (TX)\otimes \text{det}(T_{\mathscr{D}})^{\vee}).\]


The \emph{singular scheme} of $\mathscr{D}$ is defined as follows. Taking the maximal  exterior power of the dual morphism $\phi^{\vee}:\Omega^1_{X}\to T^{\vee}_{\mathscr{D}}$ and a twist, we obtain a morphism \[\Omega^s_{X}\otimes\text{det} (T_{\mathscr{D}})\to \mathcal{O}_X;\] the image of such morphism is the ideal sheaf $\mathcal{I}_{Z/X}$ of a subscheme $Z\subset X$, the singular scheme of $\mathscr{D}$.\\


\section{The Forgetful Morphism}\label{sec:VEM}
The material of this section can be found in \cite{Calvo-Andrade20209011}. We do hear for the sake of completeness.\\
Let $X$ now denote a polarized nonsingular projective variety of dimension $n$. The Grothendieck's Quot-scheme for the tangent bundle $TX$ is defined as follows.\\


Let $\mathfrak{Sch}_{/\mathbb{C}}$ denote the category of schemes of finite type over $\mathbb{C}$, and  $\mathfrak{Sets}$ be the category of sets. Fix a polynomial $P\in \mathbb{Q}[t],$ and consider the functor 
\[\mathcal{Q}uot^P\colon\mathfrak{Sch}^{\text{op}}_{/\mathbb{C}} \to \mathfrak{Sets}, \hspace{0,5cm} \mathcal{Q}uot^P(S):=\{(N,\eta)\}/\sim,\]
where
\begin{enumerate}
	\item[(i)] $N$ is  a coherent sheaf of $\mathcal{O}_{X\times S}$-modules, flat over $S$, such that the Hilbert polynomial of $N_s:=N|_{X\times\{s\}}$ is equal to $P$ for every $s\in S$;
	\item[(ii)] $\eta\colon\pi^*_XTX \to N$ is an epimorphism, where $\pi_X\colon X\times S\to X$ is the standard projection onto the first factor.
\end{enumerate}
In addition, we say that $(N,\eta)\sim(N',\eta')$ if there exists an isomorphism $\gamma\colon N\to N'$, such that $\gamma\circ\eta=\eta'$.\\

Finally, if $f\colon R\to S$ is a morphism in $\mathfrak{Sch}_{/\mathbb{C}}$, we define $\mathcal{Q}uot^P(f)\colon\mathcal{Q}uot^P(S)\to \mathcal{Q}uot^P(R)$ by $(N,\eta)\to (f^*N,f^*\eta)$. Elements of the set $\mathcal{Q}uot^P(S)$ will be denoted by $[N,\eta]$.

\begin{theorem}
	The functor $\mathcal{Q}uot^P$ is represented by a projective scheme $\mathcal{Q}^P$ of finite type over $\mathbb{C}$, that is, there exists an isomorphism of functors $\mathcal{Q}uot^P\overset{\sim}\longrightarrow \text{Hom}(\cdot,\mathcal{Q}^P)$.
\end{theorem}
The next subset of the   Quot-scheme will be important to describe the moduli space of codimension one distribution:
\begin{equation}\label{subsetquot}
\mathcal{D}^P:=\{[N,\eta]\in \mathcal{Q}^{P_{TX}-P};\ N \ \text{is torsion free}\},
\end{equation}
where $P_{TX}$ is the Hilbert polynomial of the tangent bundle. Note that, by \cite[Proposition 2.3.1]{huybrechts2010geometry}, $\mathcal{D}^P$ is an open subset of $\mathcal{Q}^{P_{TX}-P}$.\\

Again, fix a polynomial $P\in \mathbb{Q}[t]$, and consider the functor
\[\mathcal{D}ist^P\colon\mathfrak{Sch}^{\text{op}}_{/\mathbb{C}} \to \mathfrak{Sets}, \hspace{0,5cm} \mathcal{D}ist^P(S):=\{(F,\phi)\}/\sim,\]
where
\begin{enumerate}
	\item[(i)] $F$ is  a coherent sheaf of $\mathcal{O}_{X\times S}$-modules, flat over $S$, such that the Hilbert polynomial of $F_s:=F|_{X\times\{s\}}$ is equal to $P$ for every $s\in S$;
	\item[(ii)] $\phi\colon F \to \pi^*_XTX $ is a morphism of sheaves, such that $\phi_s\colon F_s\to TX$ is injective, and $\text{coker}\phi_s$ is torsion-free for every $s\in S$.
\end{enumerate}
In addition, we say that $(F,\phi)\sim(F',\phi')$ if there exists an isomorphism $\beta\colon F\to F'$, such that $\phi'\circ\beta=\phi$.

Finally, if $f\colon R\to S$ is a morphism in $\mathfrak{Sch}_{/\mathbb{C}}$, we define $\mathcal{D}ist^P(f)\colon \mathcal{D}ist^P(S)\to \mathcal{D}ist^P(R)$ by $(F,\phi)\to (f^*F,f^*\phi)$. Equivalence classes in the set $\mathcal{D}ist^P(S)$ will be denoted by $[F,\phi]$.

Each pair $(F,\phi)$ should be regarded as a \emph{family of distributions on X parametrized by the scheme} $S$ (or an $S$-family for short): for each $s\in S$, we have the distribution given by the short exact sequence
\begin{equation*}
\mathcal{D}_s: 0\to F_s \overset{\phi_s}\longrightarrow TX \overset{\eta_{\phi_s}}\longrightarrow \text{coker}\phi_s \to 0,
\end{equation*}
where $\eta_{\phi}$ denotes the canonical projection $\pi^*_XTX\to \coker\phi$. Note that if $(F,\phi)$ and $(F',\phi')$ are equivalent $S$-families of distributions on $X$, then for each $s$, the distributions $\mathcal{D}_s$ and $\mathcal{D}'_s$ are isomorphic.

\begin{proposition} 
	The functor $\mathcal{D}ist^P$ is represented by the quasi-projective scheme $\mathcal{D}^P$, described in display \eqref{subsetquot}, that is, there exists an isomorphism of functors $\mathcal{D}ist^P\overset{\sim}\longrightarrow \text{Hom}(\cdot,\mathcal{D}^P)$.
\end{proposition}

For the construction of the forgetful morphism, we need to recall the construction of the Gieseker-Maruyama moduli space of semistable sheaves on $X$; see \cite[Section 4]{huybrechts2010geometry}. Recall that a torsion-free sheaf $\mathscr{E}$ on $X$ is said to be (semi)stable if every nontrivial subsheaf $\mathscr{F}\subset \mathscr{E}$ satisfies
\[\frac{P_{\mathscr{F}}(t)}{\text{rank}\, (\mathscr{F})}(\leq)<\frac{P_{\mathscr{E}}(t)}{\text{rank}\, (\mathscr{E})},\]
for $t$ sufficiently large. 

Consider the functor
\[\mathcal{M}^P:\mathfrak{Sch}^{\text{op}}_{/\mathbb{C}} \to \mathfrak{Sets}, \hspace{0,5cm} \mathcal{M}^P(S):=\{F\}/\sim,\]
where $F$ is  a sheaf on $X\times S$, flat over $S$, such that for each $s\in S$ the sheaf $F_s:=F|_{X\times\{s\}}$ is semistable and has Hilbert polynomial equal to $P$; the equivalence relation $\sim$ is just an isomorphism of sheaves on $X\times S$.\\

The functor $\mathcal{M}^P$ admits a \emph{coarse moduli space}, here denoted by $M^P$, which is a projective scheme see \cite[Section 4]{huybrechts2010geometry}; this means that the closed points of $M^P$ are in bijection with the $S$-equivalence classes of semistable sheaves on $X$ with fixed Hilbert polynomial $P$ and that there exists a natural transformation $\mathcal{M}^P\longrightarrow\text{Hom}(\cdot, M^P)$. Below, $M^{P, st}$ will denote the open subset of $M^P$ consisting of stable sheaves.\\

Let $\mathcal{D}ist^{P,ss}$ be the subfunctor of $\mathcal{D}ist^P$ given by
\[\mathcal{D}ist^{P,ss}:\mathfrak{Sch}^{\text{op}}_{/\mathbb{C}} \to \mathfrak{Sets}, \hspace{0,5cm} \mathcal{D}ist^{P,ss}(S):=\{(F,\phi)\}/\sim,\]
where $F_s$ is now assumed to be semistable for each $s\in S$. Clearly, $\mathcal{D}ist^{P,ss}\simeq\text{Hom}(\cdot,\mathcal{D}^{P,ss})$, where
\begin{equation*}
\mathcal{D}^{P,ss}:=\{[F,\phi]\in \mathcal{D}^{P}| F \ \text{is semistable}\};
\end{equation*}
note that $\mathcal{D}^{P,ss}$ is an open subset of $\mathcal{D}^{P}$. Similarly, we will also consider the following open subset of $\mathcal{D}^{P}$:
\begin{equation*}
\mathcal{D}^{P,st}:=\{[F,\phi]\in \mathcal{D}^{P}| F \ \text{is stable}\}.
\end{equation*}

\begin{lemma}
	There exists a forgetful morphism
	\[\varpi: \mathcal{D}^{P,ss}\to M^P, \hspace{0,5cm} \varpi([\mathscr{D}]):=[T_{\mathscr{D}}].\]
	In addition, if $T_{\mathscr{D}}$ is stable and satisfies $\Ext^1(T_{\mathscr{D}}, TX)=\Ext^2(T_{\mathscr{D}}, T_{\mathscr{D}})=0$, then $[\mathscr{D}]$ is a  nonsingular point of $\mathcal{D}^{P,st}$, $\varpi$ is a submersion at $[T_{\mathscr{D}}]\in M^{P,st}$, and
	\[\dim_{[\mathscr{D}]}\mathcal{D}^{P,st}=\dim\Ext^1(T_{\mathscr{D}},T_{\mathscr{D}})+\dim\Hom(T_{\mathscr{D}},TX)-1.\]
\end{lemma}
\begin{proof}
	See \cite[Lemma 2.5]{Calvo-Andrade20209011}.
\end{proof}
 Since the tangent sheaf of a distribution is always reflexive, the image of $\varpi$ is contained in the open subset of $M^p$ consisting of reflexive sheaves. In the next lemma let $M^{P,r, st}$ denote the open subset of $M^p$ consisting of stable reflexive sheaves.
\begin{lemma} \label{el lema}
	Assume that the forgetful morphism $\varpi: \mathcal{D}^{P,st}\to M^{P,r,st}$ is surjective and  that $M^{P,r,st}$ is irreducible. If $\dim\Hom(F,TX)$ is constant for all $[F]\in M^{P,r,st}$, then $\mathcal{D}^{P,st}$ is irreducible and
	\[\dim\mathcal{D}^{P,st}=\dim M^{P,r,st}+\dim\Hom(F,TX)-1.\]
	If, in addition, $\Ext^2(T_{\mathscr{D}}, T_{\mathscr{D}})=0$, for every $[\mathscr{D}]\in \mathcal{D}^{P,st}$, then $\mathcal{D}^{P,st}$ is nonsingular and
	\[\dim\mathcal{D}^{P,st}=\dim\Ext^1(T_{\mathscr{D}}, T_{\mathscr{D}})+\dim\Hom(T_{\mathscr{D}},TX)-1.\]
\end{lemma}
\begin{proof}
	See \cite[Lemma 2.6]{Calvo-Andrade20209011}.
\end{proof}





\section{Codimension One Distributions on $\mathbb{P}^3$}\label{sec:aposteriori}

In this section, we only consider codimension one distribution on $X=\mathbb{P}^3$. The integer $d:=2-c_1(T_{\mathscr{D}})\geq 0$ is called the \emph{degree}  of the distribution $\mathscr{D}$, and $N_{\mathscr{D}}=\mathcal{I}_{Z/\p3}(d+2)$, where $Z$ is the singular scheme of $\mathscr{D}$. Therefore, sequence \ref{dis} now reads
\begin{equation}\label{dis on p3}
\mathscr{D}: 0\to T_{\mathscr{D}} \overset{\phi}\longrightarrow \tp3 \overset{\pi}\longrightarrow \mathcal{I}_{Z/\p3}(d+2) \to 0,
\end{equation}
where $T_{\mathscr{D}}$ is a rank 2 reflexive sheaf. The singular scheme $Z$ has dimension at most 1, let   $\sing_1(Z)$ be the one-dimensional component of $Z$ and $\sing_0(Z)$ be the zero-dimensional one.\\

Let $C=\sing_1(Z)$, by \cite[Theorem 3.1]{Calvo-Andrade20209011} we have
\begin{align*}
    c_1(T_\sD) &=2-d\\
    c_2(T_\sD) &=d^2+2-\deg(C)\\
    c_3(T_\sD) &=d^3+2d^2+2d-\deg(C)(3d-2)+2P_{a}(C)-2,
\end{align*}

where $c_1(T_\sD), c_2(T_\sD), c_3(T_\sD)$ are the first, second and third Chern classes of the tangent sheaf $T_\sD$.\\
For codimension one degree 3 distributions, sequence \ref{dis on p3}, and last equations now reads
\begin{equation}\label{dis 3 on p3}
\mathscr{D}: 0\to T_{\mathscr{D}} \overset{\phi}\longrightarrow \tp3 \overset{\pi}\longrightarrow \mathcal{I}_{Z/\p3}(5) \to 0,
\end{equation}
and
\begin{align*}
    c_1(T_\sD) &=-1\\
    c_2(T_\sD) &=11-\deg(C)\\
    c_3(T_\sD) &=49-7\deg(C)+2P_{a}(C).
\end{align*}

We denote the moduli space $\mathcal{D}^P$ with
\[P=2{t+3 \choose 3}+\frac{1}{2}(t+2)(t+1)(2-d)-(t+2)c+\frac{1}{2}(l+(d-2)c)\]


by $\mathcal{D}(d,c,l)$; where $P, d$ are the Hilbert polynomial and degree of the isomorphism classes of distributions, $c=c_2(T_{\mathscr{D}})$ and $l=c_3(T_{\mathscr{D}})$ are the second and third Chern classes of the tangent sheaf. Similarly, $\mathcal{D}^{st}(d,c,l):=\mathcal{D}^{P,st}$.\\

Our main tool is the forgetful morphism described in Lemma \ref{el lema}. 

\[\varpi:\mathcal{D}^{st}(d,c,l)\to \mathcal{R}(2-d,c,l), \hspace{0,5cm} \varpi([\mathscr{D}]):=[T_{\mathscr{D}}], \]
where $\mathcal{R}(2-d,c,l)$ denotes the moduli space of stable rank 2 reflexive sheaves on $\mathbb{P}^3$ with Chern classes $(c_1,c_2,c_3)=(2-d,c,l)$. with this new notation the first part of Lemma \ref{el lema} now reads.\\

\begin{theorem}\label{el teorema}
Assume that the forgetful morphism $\varpi:\mathcal{D}^{st}(d,c,l)\to \mathcal{R}(2-d,c,l)$ is surjective and that $\mathcal{R}(2-d,c,l)$ is irreducible. If $\dim\Hom(\mathcal{F},\tp3)$ is constant for all $[\mathcal{F}]\in \mathcal{R}(2-d,c,l)$, then $\mathcal{D}^{st}(d,c,l)$ is irreducible and
\[\dim\mathcal{D}^{st}(d,c,l)=\dim\mathcal{R}(2-d,c,l)+\dim\Hom(\mathcal{F},\tp3)-1.\]
\end{theorem}

\section{Castelnuovo-Mumford regularity}

The next definition is made for $\pn$ but can be made for any projective scheme over a field $k$. 

\begin{definition}
	A coherent sheaf $\mathcal{F}$ on $\pn$ is said to be $m$-regular if
	\[H^i\!\left(\mathcal{F}(m-i)\right)=0,\]
	for all $i>0$.
\end{definition} 

The next theorem is in a more general setting in \cite[Lemma 1.7.2]{huybrechts2010geometry} and the proof can be found in \cite{mumford1966lectures} or \cite{kleiman225theoremes}.

\begin{theorem}\label{teorema regularity}
	If $\mathcal{F}$ is $m$-regular, then the following holds:
	\begin{enumerate}
		\item $\mathcal{F}$ is $m'$-regular for all integers $m'\geq m$;
		\item $\mathcal{F}(m)$ is globally generated.
	\end{enumerate}
\end{theorem}

Using the last concept, the next lemma is key, showing that our space of distributions is nonempty.
\begin{lemma}\label{corolario que genera distribuciones}
	Let $G$ be a globally generated rank 2 reflexive sheaf on $\mathbb{P}^3$. Then $G^{\vee}(1)$ is the tangent sheaf of a codimension one distribution $\mathscr{F}$ of degree $c_1(G)$ with $c_2(T_{\mathscr{F}})=c_2(G)-c_1(G)+1$, and $c_3(T_{\mathscr{F}})=c_3(G)$. 
\end{lemma}
\begin{proof}
	See \cite[Appendix]{Calvo-Andrade20209011}.
\end{proof}


\section{Stability of Codimension one degree 3 Distributions on $\mathbb{P}^3$}
Let $F$ be a torsion-free coherent sheaf over $\pn$. We set 
\[\mu(F):=\frac{c_1(F)}{\text{rk}(F)},\]

where $c_1(F)$ is the first Chern class of $F$ and $\text{rk}(F)$ is the rank of $F$.

\begin{definition}
	A torsion-free coherent sheaf $E$ over $\pn$ is $\mu$-\emph{semistable} if for every coherent subsheaf $0\neq F\subset E$ we have
	\[\mu(F)\leq \mu(E).\]
If moreover for all coherent subsheaves $F\subset E$ with $0<\text{rk} (F)<\text{rk} (E)$ we have
	\[\mu(F)<\mu(E),\]
	then $E$ is $\mu$-\emph{stable}.
	\end{definition}
For rank 2 reflexive sheaves on $\p^3$, the concepts of $\mu$-stability and Gieseker Stability are equivalent.
For codimension one degree 3 distribution on $\p^3$, the concept of stability is equivalent to the concept of semi-stability, see \cite[Lemma 1.2.5, Chapter 2]{OSS}.

\begin{lemma}\label{stability}
    A reflexive sheaf $E$ of rank 2 over $\p^3$ with $c_1(E)=-1$ is stable if and only if $H^0(E)=0$.
\end{lemma}
\begin{proof}
    See \cite[Lemma 1.2.5, Chapter 2]{OSS} and notice that $E$ is already normalized.
\end{proof}

\begin{theorem}
    Let $E$ be the tangent sheaf of codimension one degree 3 distribution $\mathcal{D}$, and $C=\sing_1(Z)$, where $Z$ is the singular scheme of $\mathcal{D}$. If  $\deg(C)\leq 7$, or equivalently, $4\leq c_2(E)\leq 11$,  then $E$ is stable.
\end{theorem}
\begin{proof}
    Use $d=3$ in \cite[Proposition 6.3]{Calvo-Andrade20209011} and the fact that if $E$ does split as a sum of line bundles, then $E$ is either $\op3(-1)\oplus\op3$ or $\op3(-2)\oplus\op3(1)$. In the first case, the Chern classes of $E$ are $(-1, 0, 0)$ and in the second case are $(-1, -2, 0)$. Then the tangent sheaf $E$ of a codimension one degree 3 distribution having $4\leq c_2(E)\leq 11$, is stable.
\end{proof}

\begin{theorem}
    Let $E$ be the tangent sheaf of a codimension one degree 3 distribution $\mathcal{D}$, with second Chern class $c_2(E)=3$,  then $E$ is stable.
\end{theorem}
\begin{proof}
    Suppose that $E$ is the tangent sheaf of a codimension one degree 3 distribution
    
    \begin{equation}
 0\to E \stackrel{\omega}{\longrightarrow} \tp3 \longrightarrow \mathcal{I}_{Z/\p^3}(5) \to 0,
\end{equation}
\end{proof}
and that $E$ is not stable. By Lemma \ref{stability} $H^0 \!(E)\neq 0$. Let $\tau$ be a nontrivial section on $H^0 \!(E)$, then we have the next conmutative diagram
$$\xymatrix{
    & 0 \ar[d] & 0 \ar[d] &  & \\
    0 \ar[r] & \mathcal{O}_{\p^3} \ar[r]\ar[d] & \mathcal{O}_{\p^3}\ar[r]\ar[d] & 0 \ar[d] & \\
    0 \ar[r] & E \ar[r]\ar[d]& \tp3\ar[r]\ar[d] & \mathcal{I}_Z(5) \ar[r]\ar[d] & 0 \\
    0 \ar[r] & \mathcal{I}_S(-1) \ar[r]\ar[d]& N \ar[r]\ar[d] & \mathcal{I}_Z(5) \ar[r]\ar[d] & 0 \\
    & 0 & 0 & 0 &
}$$
 Where $S:=(\tau)_0$ and  degree of $S$ is $\text{deg}(S)=c_2(E)=3$. $\tau$ induces a \emph{foliation by curves} $\mathscr{G}$ of degree 1, see \cite{correa2023classification} for a detailed account on foliations by curves.
  \begin{equation}
 \mathscr{G}: 0\to \mathcal{O}_{\p^3} \longrightarrow \tp3 \longrightarrow N \to 0,
\end{equation}
Dualizing the last sequence, we get
\begin{equation}
 0\to N^{\vee} \longrightarrow \Omega_{\p^3} \longrightarrow \mathcal{I}_{w} \to 0,
\end{equation}
where $ \mathcal{I}_{w}$ is an ideal sheaf, and $W$ is the one-dimensional component of the singular scheme of the foliation by curves $\mathscr{G}$. Note that $S\subset W$, then $3=\text{deg}(S)\leq \text{deg}(W)$ and we have a contradiction, because by \cite[Theorem 4]{Galeano2022} $\text{deg}(W)\leq 2$.
 
\section{Classification of degree three Distributions}\label{sec:numres}

Let 
\begin{equation}
\mathscr{D}: 0\to T_{\mathscr{D}} \longrightarrow \tp3 \longrightarrow \mathcal{I}_{Z/\p3}(5) \to 0,
\end{equation}
be a codimension one degree 3 distribution on $\p^3$, let also $C$  and $\mathcal{U}$ be the one and zero dimensional components of the singular scheme of $\mathscr{D}$ respectively. By \cite[Inequallity 33]{Calvo-Andrade20209011}
\[\deg(C)\leq 13.\]

\begin{proposition}
    If $T_\sD$ is the tangent sheaf of a codimension one distribution of degree 3, then $-2\leq c_2(T_\sD)\leq 11$, or equivalently, $0\leq \deg(C)\leq 13$.
\end{proposition}
\begin{proof}
    Since  $\deg(C)\geq 0$ we have $c_2(T_\sD)\leq 11$. On the other hand by \cite[Display 34, pag 28]{Calvo-Andrade20209011} we have that $c_2(T_\sD)\geq 1-3=-2$.
\end{proof}

Note that by \cite[Corollary 2.4]{hartshorne1980stable} we have $c_1(T_\sD)c_2(T_\sD)\equiv c_3(T_\sD) \mod{2}$, then $c_2(T_\sD)+c_3(T_\sD)$ must be even, and $c_3(T_\sD)\geq 0$, see \cite{hartshorne1980stable}. If $H^0(T_\sD(-1))\neq 0$, then by \cite[Lemma 4.3]{Calvo-Andrade20209011} $T_\sD$ splits as a sum of line bundles, and the only possibilities are $\op3(-1)\oplus\op3$ or $\op3(-2)\oplus\op3(1)$. For all other cases, we consider $H^0(T_\sD(-1)) = 0$ and then we can use $c_3(T_\sD)\leq c_2^2(T_\sD)+2c_2(T_\sD)$, see \cite[Theorem 8.2. c)]{hartshorne1980stable}. Also, by \cite[Theorem 8.2. d)]{hartshorne1980stable} $c_3(T_\sD)\leq c_2(T_\sD)^2$, for $T_\sD$ stable. 

We now examine each of the possible values of $\deg(C)$, or equivalently, $c_2(T_\sD)$. 

\begin{enumerate}
    \item If $\deg(C)=0$, then $C=\emptyset$ and the singular scheme consists just on points. The Chern classes of the tangent sheaf $T_\sD$ must be $(-1, 11, 51)$, where 51 is the number of points in dimension zero. Also 
    \item If $\deg(C)=1$, we have that $C$ is a line and hence the arithmetic genus is 0. The Chern classes of the tangent sheaf must be $(-1, 10, 42)$.
    \item If $\deg(C)=2$, then by \cite[Corrolary 1.6]{N} $P_{a}(C)\leq 0$. By \cite[Corollary 1.6]{N} any such curve $C$ of arithmetic genus less than -1 is a double line, and by \cite[Corollary 9]{Galeano2022}  we have $P_a(C)\geq -3$. Then the only possibilities for third Chern classes are: 29, 31, 33, 35.
    \item If $\deg(C)=3$, then by \cite[Theorem 3.1]{Hspcurve} $P_a(C)\leq \frac{1}{2}(3-1)(3-2)=1$.
    \item If $4\leq\deg(C)\leq 13$, a similar argument to the last item bounds the arithmetic genus of $C$.
\end{enumerate}
    
We then have the next table.

\begin{table}[H] \centering
\begin{tabular}{c c c}
\hline
 deg(C) & $c_2(T_\sD)$ &$c_3(T_\sD)$   \\ \hline \hline
 0 & 11 & 51 \\
 1 & 10 &  42 \\
 2 & 9  &  29, 31, 33, 35\\ 
 3 & 8 & 0, 2, 4, 6, 8, 10, 12, 14, 16, 18, 20, 22, 24, 26, 28, 30  \\
 4 & 7 & 1, 3, 5, 7, 9, 11, 13, 15, 17, 19, 21, 23, 25, 27 \\
 5 & 6 & 0, 2, 4, 6, 8, 10, 12, 14, 16, 18, 20, 22, 24, 26 \\
 6 & 5 & 1, 3, 5, 7, 9, 11, 13, 15, 17, 19, 21, 23, 25 \\
 7 & 4 & 0, 2, 4, 6, 8, 10, 12, 14, 16 \\
 8 & 3 & 1, 3, 5, 7, 9, 11, 13, 15 \\
 9 & 2 & 0, 2, 4, 6, 8 \\
 10 & 1 & 1, 3 \\
 11 & 0 & 0 \\
 13 & -2 & 0 \\
 
\end{tabular}
\caption{Possible Chern classes of codimension one distribution of degree 3; the first column describes the degree of a generic point in the irreducible component of the Hilbert scheme that contains $\sing_1(\sD)$. 
} 
\label{deg 3 table}
\end{table}



\section{Existence of degree  3 Distributions}
\begin{theorem}
    There exists codimension one distribution with tangent sheaf having Chern classes $(-1, 10, 42)$.
\end{theorem}
\begin{proof}
    Take $d=3$ in \cite[proposition 36]{Galeano2022}.
\end{proof}

\begin{theorem}\label{Teorema 2 regular}
    There exist codimension one distributions with tangent sheaf having Chern classes $(-1, 2, 2)$, $(-1, 3, 5)$, and $(-1, 3, 7)$.
\end{theorem}
\begin{proof}
    Let $E$ be a stable reflexive sheaf with Chern classes as in the Theorem. By \cite[Table 2.6.1, Table 3.14.1, Table 3.15.1]{chang1984stable} we have that $H^1\!(E(1))=H^2\!(E)=H^3\!(E(-1))=0$, then  $E$ is 2-regular and by Theorem \ref{teorema regularity} $E(2)$ is globally generated.  By Lemma \ref{corolario que genera distribuciones} $(E(2))^{\vee}(1)=E$ is the tangent sheaf of a codimension one distributions of degree 3.
\end{proof}

\section{Moduli Spaces}

\begin{theorem}
    The moduli space $\mathcal{D}^{st}(3, 3, 5)$ of codimension one distribution of degree 3 with Chern classes $(-1, 3, 5)$ and stable and general tangent sheaf, is  an irreducible quasi projective variety, of dimension
    \[\dim\mathcal{D}^{st}(3, 3, 5)= 42.\]
\end{theorem}
\begin{proof}
     Let $\mathcal{R}(-1, 3, 5)$ denote the Moduli space of rank 2 reflexive sheaves with Chern classes $(-1, 3, 5)$ and $\mathcal{R}_0(-1, 3, 5)$ the open subset of general sheaves of $\mathcal{R}(-1, 3, 5)$.  If $[E]\in \mathcal{R}_0(-1, 3, 5)$,  by the proof of Theorem \ref{Teorema 2 regular}, $E$ is the tangent sheaf of a codimension one degree 3 distribution on $\p^3$, then the forgetful morphism $$\varpi:\mathcal{D}^{st}(3, 3, 5)\to \mathcal{R}_0(-1, 3, 5),$$ is surjective. Also, note that $\mathcal{R}(-1, 3, 5)$ is irreducible of dimension 19, see \cite[Theorem 3.14]{chang1984stable}.\\

     By the proof of \cite[Theorem 3.14]{chang1984stable} we have the next exact sequences
     \begin{equation}\label{exact sequence (-1, 3, 5)}
         0\to E'\to \op3^{\oplus 6}\to E(2)\to 0,
     \end{equation}

\begin{equation}\label{exact sequence (-1, 3, 5)2}
         0\to \op3\to E'(1)\to \tp3(-1)\to 0,
     \end{equation}

where $E'$ is a rank 4 vector bundle. Tensoring by $\tp3(-1)$ the exact sequence \ref{exact sequence (-1, 3, 5)}, we obtain
\begin{equation}
         0\to E'\otimes\tp3(-1)\to \tp3(-1)^{\oplus 6}\to E^{\vee}\otimes\tp3\to 0,
     \end{equation}
now taking cohomology we have the next exact sequence
\begin{equation*}
    0\to H^0(E'\otimes\tp3(-1))\to H^0(\tp3(-1))^{\oplus 6}\to H^0(E^{\vee}\otimes\tp3)\to H^1(E'\otimes\tp3(-1))\to 0
\end{equation*}

Tensoring the Euler sequence
\[0\to \op3\to \op3(1)^{\oplus 4}\to \tp3\to 0,\]
by $E'(-1)$, taking cohomology,  using also exact sequence \ref{exact sequence (-1, 3, 5)2}, we get that 
\[H^0(E'\otimes\tp3(-1))=H^1(E'\otimes\tp3(-1))=\{0\}.\]
Then 
\begin{align*}
    \dim\Hom(E,\tp3)= \dim H^0(E^{\vee}\otimes\tp3) &=6\dim H^0(\tp3(-1))\\
                                                    &= 24,
\end{align*}

is constant for every $[E]\in \mathcal{R}_0(-1, 3, 5)$. Then by Theorem \ref{el teorema}, $\mathcal{D}^{st}(3, 3, 5)$ is irreducible and 
\[\dim \mathcal{D}^{st}(3, 3, 5)= 19+ \dim\Hom(E,\tp3)-1=42.\]

\end{proof}

\section*{Declarations}
\textbf{Conflict of interest}  None\\
\textbf{Availability of data and materials} There is no data inclede in this paper. The data availability is not applicable.





\bibliographystyle{plain}
\bibliography{references}
\end{document}